\documentclass[a4paper,12pt,reqno]{amsart}
\allowdisplaybreaks
\usepackage{amsmath,amssymb}%
\usepackage{exscale}%
\usepackage{amsfonts}%
\usepackage{palatino}
\usepackage{amsthm}
\usepackage{epsfig}
\usepackage{xcolor}
\usepackage{hyperref}
\usepackage[utf8]{inputenc}
\usepackage[english]{babel}
\usepackage{mathtools}
\usepackage{cite}

\usepackage{xcolor}
\usepackage[ left=2.8cm,
right=2.8cm,
top=3.5cm,
bottom=3.5cm]{geometry}

\usepackage{cleveref}
\usepackage[footnote,draft,silent,nomargin]{fixme}

\hypersetup{%
	plainpages=false,%
	pdfauthor={Author(s)},%
	pdftitle={Title},%
	pdfsubject={Subject},%
	bookmarksnumbered=true,%
	colorlinks,%
	citecolor=black,%
	filecolor=black,%
	linkcolor= black,
	urlcolor=black,%
	pdfstartview=FitH%
}

\let\epsilon\varepsilon

\let\phi\varphi

\let\theta\vartheta

\newtheorem{mytheorem}{Theorem}[section]
\newtheorem{myprop}[mytheorem]{Proposition}
 
\theoremstyle{definition}
\newtheorem{mydef}[mytheorem]{Definition}

\newtheorem{myre}[mytheorem]{Remark}
\newtheorem{mylemma}[mytheorem]{Lemma}

\newcommand{\R}{\mathbb{R}}

\newcommand{\N}{\mathbb{N}}

\newcommand{\cA}{\mathcal{A}}
\newcommand{\cQ}{\mathcal{Q}}
\newcommand{\cS}{\mathcal{S}}
\newcommand{\ca}{{a}}

\newcommand{\ce}{\mathbf{e}}
\newcommand{\cn}{{n}}

\newcommand{\cv}{{v}}
\newcommand{\Span}{\operatorname{span}}
 \newcommand{\bN}{\mathbb{N}}

\newcommand{\norm}[1]{\left\lVert #1\right\rVert}

\newcommand{\bR}{\mathbb{R}}
\newcommand{\bZ}{\mathbb{Z}}

\newcommand{\supp}{\operatorname{supp}}
\newcommand{\qp}{\mathcal{Q}}

\newcommand{\ptt}{\mathbb{I}}
\newcommand{\brac}[1]{\langle #1\rangle}

\newcommand{\ve}{\varepsilon}
\renewcommand{\epsilon}{\varepsilon}

\begin{document}
\title[Unconditional bases for homogeneous $\alpha$-modulation spaces ]{Unconditional bases for homogeneous\\ $\alpha$-modulation type spaces } \author{Morten Nielsen }   \date{\today}
\begin{abstract}
In this article we construct orthonormal bases compatible with bi-variate homogeneous $\alpha$-modulation
spaces and  the associated spaces of Triebel-Lizorkin type. The construction is based on generating a separable $\alpha$-covering and using
carefully selected tensor products of univariate brushlet functions with regards to this covering.
We show that the associated systems form an unconditional bases for the homogeneous $\alpha$-spaces  of Triebel-Lizorkin type.
\end{abstract}
\maketitle

\section{Introduction}

Unconditional bases for smoothness spaces play an important role for many applications  as the bases often provide simple characterizations of the space in terms of certain sparseness conditions. For example, smoothness measured in a Besov space is equivalent to a certain sparseness of a wavelet expansion \cite{Meyer1992}. Moreover, norm characterizations often allow us to identify certain smoothness spaces as nonlinear approximation spaces \cite{Gribonval2004,Kyriazis2002}. As a consequence we gain a better understanding of how to compress smooth functions by using the  sparse representation of the function in the unconditional basis \cite{DeVore1992,DeVore1992a}.

The $\alpha$-modulation spaces $M^{s,\alpha}_{p,q}(\bR^d)$, $\alpha\in
[0,1]$, form a parameterized family of smoothness spaces defined on $\bR^d$ that
include the Besov and modulation spaces as special cases, corresponding
to $\alpha=1$ and $\alpha=0$, respectively.  The spaces are built from
the same type of scheme arising from different segmentations of the
frequency space. The $\alpha$-parameter determines the nature of the
segmentation. For example, the Besov spaces ($\alpha=1$) correspond to
a dyadic segmentation of the frequency space, while the modulation
spaces ($\alpha=0)$ correspond to a uniform covering. The intermediate
cases correspond to ``polynomial type'' segmentations of the frequency
space. The classical $\alpha$-modulation spaces are inhomogeneous spaces in the sense that the underlying segmentation of the
frequency space cover the zero-frequency (so a natural low-pass filter is integrated in the representation). Recently, however, the $\alpha$-modulation spaces have been extended to a homogeneous setup. The main contribution of of this note is to present a construction of unconditional bases for homogeneous $\alpha$-modulation spaces.

The $\alpha$-modulation spaces were introduced by Gr\"obner
\cite{Groebner1992}, and it was pointed out by 
Feichtinger and Gr\"obner
\cite{Feichtinger1987,Feichtinger1985} that Besov and modulation spaces are
special cases of an abstract construction, the so-called decomposition
type Banach spaces. The coverings giving rise to
$\alpha$-modulation spaces have also been considered 
by P{\"a}iv{\"a}rinta and Somersalo in \cite{Paeivaerinta1988} as a tool to study
 pseudo-differential operators. The close connection between decomposition spaces and  classical smoothness space such as modulation spaces was  first pointed out by Triebel \cite{Triebel1983a}. Triebel's work  later inspired a more general treatment of decomposition smoothness spaces  \cite{Borup2008,Borup2007}. Another benefit of the connection to the general theory of decomposition spaces it that one can easily construct associated smoothness spaces of Triebel-Lizorkin type ($\alpha$-TL spaces).

The main contribution of the present paper is to offer a
construction of an orthonormal basis for $L_2(\bR^2)$ that extends to an unconditional basis for bi-variate
$\alpha$-modulation spaces and for the associated $\alpha$-TL spaces.We believe that  our construction is the first
example of a non-redundant representation system for multivariate
$\alpha$-TL spaces. Orthonormal bases for classical (inhomogeneous) $\alpha$-modulation spaces were constructed by the author in \cite{Nielsen2010}. This construction was later extended to anisotropic bi-variate setting by Rasmussen \cite{Rasmussen2012}.

The orthonormal basis is constructed using a carefully calibrated
tensor product approach based on so-called univariate brushlet
systems.  Brushlets are the image of a local trigonometric basis
under the Fourier transform, and such systems were introduced by Laeng
\cite{Laeng1990}. Later Coifman and Meyer \cite{Meyer1997a} used
brushlets as a tool for image compression.  In \cite{Borup2003}, Borup and
Nielsen used the freedom to choose the frequency localization of a
brushlet system to construct (orthonormal) unconditional brushlet
bases for the univariate $\alpha$-modulation spaces. 
Using the orthonormal basis for bi-variate
$\alpha$-modulation spaces, we give a    
characterization of the bi-variate $\alpha$-modulation spaces in terms
a sparseness condition on the expansion coefficients, and we  also 
identify the $\alpha$-modulation spaces
as approximation spaces associated with nonlinear $m$-term
approximation.

\section{Bi-variate brushlet bases} 
Given an orthonormal basis $\{f_k\}_k$ for $L_2(\bR)$, a universal method to created an associated orthonormal basis for $L_2(\bR^2)$ is to consider the  tensor product basis $\{f_k\otimes f_{k'}\}_{k,k'}$. While this works very well for e.g.\ the trigonometric system on a cube, the straightforward tensor product approach can be considered more problematic for wavelets and similar systems as basis elements with long "skinny" support in the frequency plane are created. Such elements are not well-adapted for analysis of classical isotropic smoothness spaces such as Besov or Triebel-Lizorkin spaces.

In this section we wish to avoid creating elements with "skinny" support in frequency, but still use a tensor product construction to obtain an orthonormal basis for $L_2(\bR^2)$. We will accomplish this by modifying the tensor product construction carefully by keeping track of the shape of the system in the frequency plane by extracting subsystems from a sequence of so-called univariate brushlet bases. We also mention that the analysis later in the paper would have been much simplified if one could have used localised orthonormal  exponential basis. But, unfortunately, this is not possible due to the Balian-Low theorem.  

To keep the notation manageable, we consider only the bi-variate case in this note, but the
reader can verify that the basic idea behind the construction can be adapted to the
general multivariate case.

\subsection{Univariate brushlets} We begin by introducing brushlet in a univariate setting. 
Each univariate brushlet basis is associated with a partition of the
frequency axis. The partition can  be chosen with almost
no restrictions, but in order to have good properties of the
associated basis we need to impose some growth conditions on the
partition. 
We have the following definition.

\begin{mydef}
A family $\ptt$ of intervals is called a {\it disjoint covering} of $\R$ if
it consists of a countable set of pairwise disjoint half-open intervals
$I=[\alpha_I,\alpha'_I)$, $\alpha_I<\alpha'_I$, such
that $\cup_{I\in \ptt}I=\R$. If, furthermore, each interval in $\ptt$
has a unique adjacent interval in $\ptt$ to the left and to the right,
and there exists a constant $A>1$ such that
\begin{equation}\label{eq:growthcond}
A^{-1}\leq \frac{|I|}{|I'|} \leq A,\qquad \text{for
  all adjacent}\; I,I'\in \ptt,
\end{equation}
we call $\ptt$ a {\it moderate disjoint covering} of $\R$.
\end{mydef}

Given a moderate disjoint covering $\ptt$ of $\R$, assign to each interval $I\in \ptt$ a cutoff
radius $\varepsilon_I>0$ at the left endpoint and a cutoff radius
$\varepsilon'_I>0$ at the right endpoint, satisfying
\begin{equation}\label{eq:epsilon}
\begin{cases}
\text{(i)}& \ve'_I= \ve_{I'}\; \text{whenever}\;
\alpha'_I=\alpha_{I'}\\
\text{(ii)}& \ve_I+\ve'_I\leq |I|\\
\text{(iii)}& \ve_I\geq c |I|,
\end{cases}
\end{equation}
with $c>0$ independent of $I$.

We are now ready to define the brushlet system. For each $I\in \ptt$,
we 
will  construct a smooth bell function localized in a
neighborhood of this interval. Take a
non-negative ramp function $\rho \in C^\infty(\R)$ satisfying
\begin{equation}\label{eq:ramp}
\rho(\xi)=\left\{ \begin{array}{ll}
0&\mbox{for}\; \xi \leq -1,\\
1&\mbox{for}\; \xi \geq 1,
\end{array} \right.
\end{equation}
with the property that
\begin{equation}\label{eq:ramp2}
\rho (\xi)^2+\rho(-\xi)^2=1\qquad \text{for all}\; \xi \in\R.
\end{equation}
Define for each $I = [\alpha_I,\alpha'_I)\in \ptt$ the {\it bell function}
\begin{equation}\label{eq:bell}
b_I(\xi) := \rho \bigg( \frac{\xi-\alpha_I}{\ve_I}\bigg)\rho
\bigg( \frac{\alpha'_I-\xi}{\ve'_I}\bigg) .
\end{equation}
Notice that $\supp (b_I) \subset [\alpha_I-\ve_I,\alpha'_I+\ve'_I]$ and
$b_I(\xi)=1$ for $\xi \in [\alpha_I+\ve_I,\alpha'_I-\ve'_I]$.
Now the set of local cosine functions
\begin{equation}\label{eq:brush1}
\hat{w}_{n,I}(\xi) =\sqrt{\frac{2}{|I|}} b_I(\xi)
\cos\biggl( \pi \bigl( n+{\textstyle \frac12}\bigr)
\frac{\xi- \alpha_I}{|I|} \biggr) ,\quad n\in \N_0,\quad I\in \ptt,
\end{equation}
 with $\bN_0:=\bN\cup \{0\}$, constitute an orthonormal basis for $L_2(\bR)$, see e.g.\
\cite{Auscher1992}. We 
call the collection $\{ w_{n,I}\colon I\in\ptt, n\in\N_0\}$ a
{\it brushlet system}. The brushlets also have an explicit
representation in the time domain. Define the set of {\it central
  bell functions} $\{ g_I\}_{I\in \ptt}$ by
\begin{equation}\label{eq:gb}
\hat{g}_I(\xi) := \rho \biggl(
\frac{|I|}{\ve_I} \xi\biggr) \rho \biggl( \frac{|I|}{\ve'_I} (1-\xi)\biggr),
\end{equation}
such that $b_I(\xi) = \hat{g}_I\bigl(|I|^{-1}(\xi -\alpha_I)\bigr)$,
 and let for notational convenience
$$e_{n,I}:= \frac{\pi \bigl( n+{\textstyle
    \frac12}\bigr)}{|I|},\qquad
I\in \ptt,\; n\in \N_0.$$
Then,
\begin{equation}
w_{n,I}(x)
= \sqrt{\frac{|I|}{2}} e^{i\alpha_Ix} \bigl\{ g_I\bigl(
  |I|(x +e_{n,I}) \bigr) + g_I\bigl( |I|(x -e_{n,I}) \bigr)
  \bigr\}.\label{eq:gw}
\end{equation}

By a straight forward calculation it can be verified (see
\cite{Borup2003}) that for $r\geq 1$ there exists a constant $C:=C(r)<\infty$, independent
of $I\in \ptt$, such that  
\begin{equation}\label{eq:gdecay}
|g_I(x)|\leq C(1+|x|)^{-r}.
\end{equation}  Thus a
brushlet $w_{n,I}$ essentially consists of two well localized humps at
the points $\pm e_{n,I}$.

Given a bell function $b_I$, define an
operator $\mathcal{P}_I:L_2(\bR) \rightarrow L_2(\bR)$ by
\begin{equation}
  \label{eq:proje}
  \widehat{\mathcal{P}_If}(\xi) := b_I(\xi)\bigl[ b_I(\xi)\hat{f}(\xi) +
b_I(2\alpha_I-\xi)\hat{f}(2\alpha_I-\xi)- b_I(2\alpha'_I-\xi)\hat{f}(2\alpha'_I-\xi)\bigr].
\end{equation}
It can be verified that $\mathcal{P}_I$ is an orthogonal projection, mapping
$L_2(\bR)$ onto $\overline{\Span}\{w_{n,I}\colon n\in \N_0\}$.  In
Section \ref{s:onb}, we will need some of the finer properties of the
operator given by \eqref{eq:proje}. Let us list properties here, and
refer the reader to \cite[Chap.\ 1]{Hernandez1996} for a more detailed
discussion of the properties of local trigonometric bases.

 Suppose $I = [\alpha_I,\alpha'_I)$
and $J=[\alpha_J,\alpha_J')$ are two adjacent compatible intervals
(i.e., $\alpha'_I=\alpha_J$ and $\epsilon_I'=\epsilon_J$). Then it holds true that
\begin{equation}
  \label{eq:summm}
  \widehat{\mathcal{P}_If}(\xi)+\widehat{\mathcal{P}_Jf}(\xi)=\hat{f}(\xi),\qquad
\xi\in[\alpha_I+\epsilon_I,\alpha_J'-\epsilon_J'],\qquad f\in L_2(\bR).
\end{equation}
We can verify \eqref{eq:summm} using the fact that
$b_I\equiv 1$ on $[\alpha_I+\epsilon_I,\alpha'_I-\epsilon_I']$ and that
$b_J\equiv 1$ on $[\alpha_J+\epsilon_J,\alpha_J'-\epsilon_J']$,
together with the fact that 
$$\text{supp}\big(b_I(\cdot)b_I(2\alpha_I-\cdot)\big)\subseteq
[\alpha_I-\epsilon_I,\alpha_I+\epsilon_I]$$
and
$$\text{supp}\big(b_I(\cdot)b_I(2\alpha_I'-\cdot)\big)\subseteq
[\alpha_I'-\epsilon_I',\alpha_I'+\epsilon_I'].
$$
 For 
$\xi\in[\alpha_I'-\epsilon_I',\alpha_J+\epsilon_J]$ we notice that 
\begin{align}
\widehat{\mathcal{P}_If}(\xi)+\widehat{\mathcal{P}_Jf}(\xi)=
[&b_I^2(\xi)+b_J^2(\xi)(\xi)]\hat{f}(\xi)\notag \\&+
\label{eq:add} b_J(\xi)b_J(2\alpha'-\xi)\hat{f}(2\alpha'-\xi)-
b_I(\xi)b_I(2\alpha'-\xi)\hat{f}(2\alpha'-\xi).  
\end{align}
We can then conclude that \eqref{eq:summm} holds true  using the
following facts (see \cite[Chap.\ 1]{Hernandez1996})
$$b_I(\xi)=b_J(2\alpha_I'-\xi),\qquad b_J(\xi)=b_I(2\alpha_J'-\xi),\qquad
\text{for }\xi\in [\alpha_I'-\epsilon_I',\alpha_J+\epsilon_J],$$
and
$$b_I^2(\xi)+b_J^2(\xi)=1,\qquad \text{for } \xi\in [\alpha_I+\epsilon_I,\alpha_J'-\epsilon_J'].$$
Moreover, 
\begin{equation}\label{eq:proj}\mathcal{P}_I+\mathcal{P}_J=\mathcal{P}_{I\cup J}\end{equation} with the $\epsilon$-values 
$\epsilon_I$ and $\epsilon_J'$ for $I\cup J$.

Finally, for a rectangle $Q=I\times J\subset \bR^2$ with 
$I = [\alpha_I,\alpha'_I)$
and $J=[\alpha_J,\alpha_J')$, 
 we define
 $P_Q=\mathcal{P}_I\otimes \mathcal{P}_J$. Clearly, $P_Q$
 is a projection operator $P_Q\colon L_2(\bR^2) \rightarrow
 \overline{\operatorname{span}}\{w_{i,I}\otimes w_{j,J}\colon i,j\in \bN_0\}.$

Notice that, 
\begin{equation}
  \label{eq:operatorP}
  P_Q = b_Q(D)\Bigl[
(\operatorname{Id}+R_{\alpha_I}-R_{\alpha_{I}'})
\otimes (\operatorname{Id}+R_{\alpha_J}-R_{\alpha_{J}'})
\Bigr]b_Q(D),
\end{equation}
where 
\begin{equation}\label{eq:S}
\widehat{b_Q(D)f} := b_Q\hat{f},
\end{equation}
with  $b_Q:=b_I\otimes b_J$, and $R_af(x) := e^{i2a}f(-x)$, 
$x,a\in \bR$. The corresponding orthonormal tensor product basis of brushlets is given by
$$w_{n,Q}:=w_{n_1,I}\otimes w_{n_2,J},\qquad n=(n_1,n_2)\in \bN_0^2.$$

\subsection{Structured \texorpdfstring{$\alpha$}{alpha}-coverings and bi-variate brushlet systems}\label{s:onb}
We now turn to the task of creating bi-variate systems with a very specific time-frequency structure that turns out to be well-adapted for the analysis of homogeneous $\alpha$-modulation spaces. We are going to fix $0\leq \alpha<1$.  In this section, $\alpha$ can be considered just a parameter that can be used to "tune" the specific time-frequency properties of the resulting bi-variate system. We first consider the following subsets of the real axis, with endpoints that are compatible with standard univariate $\alpha$-coverings, see \cite{Borup2006b}, 
$$A_j:= \big[-j^{\frac{1}{1-\alpha}},j^{\frac{1}{1-\alpha}}\big),\qquad j=1,2,\ldots$$
For the low frequencies we will need the following subsets
$$A_j:= \big[-|j|^{-\frac{1}{1-\alpha}},|j|^{-\frac{1}{1-\alpha}}\big), \qquad j=-1,-2,\ldots$$
We will need to create additional intervals for the final covering. For this we make a further 
subdivision of $[-j^{\frac{1}{1-\alpha}},j^{\frac{1}{1-\alpha}}]$ into $2j/r_1+1$ intervals, where $r_1$ is chosen sufficiently small and such that 
$2|j|/r_1=2N_j\in 2\bN$. We write
$$[-j^{\frac{1}{1-\alpha}},j^{\frac{1}{1-\alpha}}]=I_{j,-N_j}\cup I_{j,-N_j+1}\cup\cdots\cup I_{j,N_j},$$
where $I_{j,n}:=[r_{j,n},r_{j,n+1}]$, $j=-N_j,\ldots,N_j$, and 
 we impose the particular "endpoint"-choices  $r_{j,-N_j+1}=-(j-1)^{\frac{1}{1-\alpha}}$ and $r_{j,N_j-1}=(j-1)^{\frac{1}{1-\alpha}}$, i.e.,
\begin{equation}\label{eq:compatible}
  I_{j,-N_j}=[-j^{\frac{1}{1-\alpha}},-(j-1)^{\frac{1}{1-\alpha}}] \text{ and }I_{j,N_j}=[(j-1)^{\frac{1}{1-\alpha}},j^{\frac{1}{1-\alpha}}].
 \end{equation}
This is done to ensure seamless "gluing" later on when we create bi-variate systems.  
We now repeat the process for $[-|j|^{-\frac{1}{1-\alpha}},|j|^{-\frac{1}{1-\alpha}}]$, $j\in\{-1,-2,\ldots\}$, and make a division into 
into $2|j|/r_1+1$ intervals.
We write, for $j\in\{-1,-2,\ldots\}$,
$$[-|j|^{-\frac{1}{1-\alpha}},|j|^{-\frac{1}{1-\alpha}}]=I_{j,-N_j}\cup I_{j,-N_j+1}\cup\cdots\cup I_{j,N_j},$$
where $I_{j,n}:=[r_{j,n},r_{j,n+1}]$, $j=-N_j,\ldots,N_j$, and we again impose particular "endpoint"-choices  $r_{j,-N_j+1}=-(|j|+1)^{-\frac{1}{1-\alpha}}$ and $r_{j,N_j-1}=(| j|+1)^{-\frac{1}{1-\alpha}}$, i.e.,
 \begin{equation}\label{eq:compatible2}
 I_{j,-N_j}=[-|j|^{-\frac{1}{1-\alpha}},-(|j|+1)^{-\frac{1}{1-\alpha}}] \text{ and }I_{j,N_j}=[(|j|+1)^{-\frac{1}{1-\alpha}},|j|^{-\frac{1}{1-\alpha}}].
  \end{equation}
For each interval $I=[r_{j,s},r_{j,s+1}]$, we associate a corresponding brushlet system with left $\epsilon$-value 
$\frac{1}{100}|r_{j,s}|^\alpha$ and right $\epsilon$-value $\frac{1}{100}|r_{j,s+1}|^\alpha$. The scaling factor  $\frac{1}{100}$ has been chosen to ensure that \eqref{eq:epsilon} is satisfied.

We now consider the rectangular ``annuli'' given by
$$\mathcal{A}_j=\mathcal{A}_j^L\cup \mathcal{A}_j^R\cup \mathcal{A}_j^T\cup \mathcal{A}_j^B,\qquad  j\in \bZ\backslash \{0\},$$
with
$$\mathcal{A}_j^L=\left\{ I_{j,-N_j}\times I_{j,n}\right\}_{-N_j \leq n\leq N_j },\qquad \mathcal{A}_j^R=\left\{ I_{j,N_j}\times I_{j,n}\right\}_{-N_j \leq n\leq N_j }$$
and
$$\mathcal{A}_j^T=\left\{ I_{j,n}\times I_{j,N_j}\right\}_{-N_j < n< N_j },\qquad \mathcal{A}_j^B=\left\{ I_{j,n}\times I_{j,-N_j}\right\}_{-N_j < n< N_j }.$$
For notational convenience, we put $\mathcal{A}_0:=\emptyset$. The following result confirms that one can build orthonormal bi-variate brushlet bases based on the covering of $\bR^2$ given by the sets $\{\mathcal{A}_j\}$.  
\begin{myprop}
The system $\{w_{n,Q}: n\in \bN_0^2, Q\in \mathcal{A}_j, j\in \bZ\}$ forms an orthonormal basis for $L_2(\bR^2)$.
\end{myprop}
\begin{proof}
We first consider orthonormality. Let $\mathcal{S}_j:=\{w_{n,Q}:Q\in \mathcal{A}_j,n\in\bN_0^2\}$ and notice that the functions in $\mathcal{S}_m$ and $\mathcal{S}_n$ have disjoint frequency support for $|m-n|>1$ (except in the special case $m=-1$ and $n=1$, which will be considered below). 
Also notice, using the separable structure of the bi-variate brushlet functions,  that the particular compatible endpoint structure, see \eqref{eq:compatible} and \eqref{eq:compatible2}, imposed on the partitioning of the sets $A_j$ ensures that $\mathcal{S}_m$ is orthogonal to $\mathcal{S}_{m+1}$ for $m<-1$ and $m\geq 1$, and using the compatibility at the frequency one, we see that $\mathcal{S}_{-1}$ is orthogonal to $\mathcal{S}_{1}$. Within each $\mathcal{S}_j$, orthonormality follows directly from the separable structure of the bi-variate brushlet system and the orthogonality of the respective univariate brushlet systems. 

We will now verify completeness of the system, where  we  first notice that for $f\in L_2(\bR^2)$, 
$$\sum_{Q\in \mathcal{A}_j, j\in \bZ} \widehat{P_Q f}(\xi),$$
is well-defined and converges pointwise as the support of each term  in the sum overlaps with at most eight other terms corresponding to adjacent rectangles.  
Next, we notice that by repeated use of \eqref{eq:proj}, for $j>1$,

$$\sum_{Q\in \mathcal{A}_j^L} P_Q=P_{I_{j,-N_j}}\otimes P_{A_j},\qquad \sum_{Q\in \mathcal{A}_j^R} P_Q=P_{I_{j,N_j}}\otimes P_{A_j},$$
$$\sum_{Q\in \mathcal{A}_j^T} P_Q=P_{A_{j-1}} \otimes P_{I_{j,N_j}},\qquad \sum_{Q\in \mathcal{A}_j^B} P_Q=P_{A_{j-1}} \otimes P_{I_{j,-N_j}},$$
So, in particular, again using \eqref{eq:proj},
$$P_{A_{j-1}\times A_{j-1}}+\sum_{Q\in \mathcal{A}_j^T\cup \mathcal{A}_j^T} P_Q=P_{A_{j-1}\times A_{j}}.$$
Using a similar argument for $ \mathcal{A}_j^L\cup \mathcal{A}_j^R$, and collecting terms,  we may conclude that
$$P_{A_{j-1}\times A_{j-1}}+\sum_{Q\in \mathcal{A}_j} P_Q=P_{A_{j}\times A_{j}}, \qquad j>1.$$
Using a parallel argument, we may conclude that
$$P_{A_{j-1}\times A_{j-1}}+\sum_{Q\in \mathcal{A}_j} P_Q=P_{A_{j}\times A_{j}}, \qquad j\leq -1.$$
Noting that $A_{-1}=A_1$, it follows that
$$\sum_{-N \leq j \leq N}\sum_{Q\in \mathcal{A}_j} P_Q=P_{A_{N}\times A_{N}}-P_{A_{-N-1}\times A_{-N-1}}.$$
It thus follows easily  that 
$$\lim_{N\rightarrow \infty}\sum_{-N \leq j \leq N}\sum_{Q\in \mathcal{A}_j} P_Q=Id_{L_2(\bR^2)},$$
in the strong operator topology. This completes the proof.
\end{proof}

Let us conclude this section by introducing some additional notation. Put $$\mathcal{Q}^\alpha:=\bigcup_{j\in\bZ} \mathcal{A}_j.$$
For an arbitrary rectangle $Q=I\times J\in \mathcal{Q}^\alpha$, we let $\xi_Q\in Q$ denote the mid-point of $Q$. Put $\mathcal{Q}_0:=[-1/2,1/2]^2$, and note that 
\begin{equation}\label{eq:affine}Q=\delta_Q(\mathcal{Q}_0)+\xi_Q,\end{equation}
where $\delta_Q:=\text{diag}(|I|,|J|)$. This shows that $\mathcal{Q}^\alpha$ is a so-called structured covering in the terminology of \cite{AlJawahri2019}. Notice also that the covering satisfies the geometric rule
\begin{equation}\label{eq:geom}
    |Q|\asymp |\xi_Q|^{\beta(Q)},\qquad \beta(Q)=\begin{cases}2\alpha,&Q\in \cup_{j>0} \mathcal{A}_j\\
2(2-\alpha),&Q\in \cup_{j<0} \mathcal{A}_j\end{cases}.
\end{equation}
\section{Homogeneous \texorpdfstring{$\alpha$}{alpha}-modulation spaces}
In this section we introduce $\alpha$-modulation and $\alpha$-TL spaces in  the homogeneous setting, an extension that was first considered in \cite{AlJawahri2019}.
In the inhomogeneous setup on the real line, an $\alpha$-covering can easily be obtained from the knots $\pm n^\beta$, $n\in\mathbb{N}$, taking $\beta=1/(1-\alpha)$ for $0\leq \alpha<1$, while in the limiting (Besov) case $\alpha=1$, we simply use dyadic knots $\pm 2^j$, $j\in\mathbb{N}$. Now, in the Besov case, we add the low frequency knots $\pm 2^{-j}$, $j\in\mathbb{N}$, to obtain a full decomposition  yielding homogeneous Besov spaces. Notice that the low frequency knots can be considered the image under $\xi\rightarrow 1/\xi$ of the high frequency knots. The idea is now to copy this process for the $\alpha$-covering obtaining low-frequency knots $\pm n^{-\beta}$, $n\in\mathbb{N}$, that can be seen to satisfy the geometric ``rule'' $|n^{-\beta}-(n+1)^{-\beta}|\asymp n^{-\beta(2-\alpha)}$, while the high-frequency knots satisfy $|(n+1)^\beta-n^\beta|\asymp n^{\alpha\beta}$. 

Inspired by these considerations, we define a general hybrid weight $\tilde{h}_\alpha:\bR^2\rightarrow \bR_+$  by $\tilde{h}_\alpha(\xi) := \rho(\xi) h_1(\xi) + (1-\rho(\xi))h_2(\xi)$, where $\rho:\bR^2\rightarrow \bR_+$ is a  smooth function that satisfies
$$\rho(x):=\begin{cases}1,&|x|\leq \frac23\\ 0,&|x|\geq \frac 43\end{cases},$$
 $h_1(\xi) = |\xi|^{2-\alpha}$ and $h_2(\xi) = |\xi|^{\alpha}$.

We now introduce the notion of an $\alpha$-covering. 
\begin{mydef}\label{def:1}
A countable set $\qp$ of subsets $Q\subset \bR^2\backslash \{0\}$ is called an
admissible covering if $\bR^2\backslash\{0\}=\cup_{Q\in \qp} Q$ and there
exists $n_0<\infty$ such that 
  $$\#\{Q'\in\qp:Q\cap Q'\not=\emptyset\}\leq n_0$$ 
  for all $Q\in\qp.$ 
Let 
\begin{align*}
  r_Q &= \sup\{ r\in\bR_+\colon B(c_r,r)\subset Q\; \text{for some}\; c_r\in \bR^2\},\\
  R_Q &= \inf\{ R\in\bR_+\colon Q\subset B(c_R,R)\; \text{for some}\; c_R\in \bR^2\}
\end{align*}
denote, respectively, the radius of the inscribed and circumscribed
disc of $Q\in \qp$. 
An admissible covering is called a homogeneous $\alpha$-covering, $0\leq
  \alpha\leq 1$, of $\bR^2$ if
$|Q|^{1/2}\asymp \tilde{h}_\alpha(\xi)$ (uniformly) for all $x\in Q$
and for all $Q\in\qp$, and there exists 
a constant $K\geq 1$ such that $R_Q/r_Q\leq K$ for
all $Q\in \qp$. 
\end{mydef}

We have already noticed that the covering $\mathcal{Q}^\alpha$ satisfies, for $Q\in \mathcal{Q}^\alpha$,
\begin{equation}\label{eq:Qrule}
|Q|^{1/2}\asymp \tilde{h}_\alpha(\xi),\qquad\text{(uniformly) for all }\xi\in Q.
\end{equation}
It is also straightforward to verify that the inscribed/circumscribed disc condition is satisfies as the rectangles in $\mathcal{Q}^\alpha$ have eccentricity close to one, so we may pick $r_Q\asymp R_Q\asymp |Q|^{1/2}$, uniformly in $Q\in \mathcal{Q}^\alpha$. We  conclude that $\mathcal{Q}^\alpha$ is indeed a homogeneous $\alpha$-covering. It was proven in \cite[Lemma 2.8]{AlJawahri2019} that the weight $\tilde{h}_\alpha$ is moderate relative to $\mathcal{Q}^\alpha$ in the sense that there exists a constant $R>0$ depending only on $\mathcal{Q}^\alpha$ such that for $Q\in \mathcal{Q}^\alpha$ 
\begin{equation}\label{eq:moderation}
    R^{-1}\leq \frac{\tilde{h}_\alpha(x)}{\tilde{h}_\alpha(y)}\leq R,\qquad x,y\in Q..
\end{equation}

In order to define smoothness spaces adapted to $\alpha$-coverings, we need to consider an associated slightly expanded $\alpha$-covering defined by the sets
$$Q^e:=\delta_Q([-0.6,0.6]^2)+\xi_Q,\qquad Q\in \mathcal{Q}^\alpha,$$
where $\xi_Q$ and $\delta_Q$ are defined in \eqref{eq:affine}. The only important characteristic of the number 0.6 in this context is that it is slightly larger than $1/2$. It is proven in \cite[Proposition 2.5]{AlJawahri2019} that  one can create two bounded partitions of unity (BAPUs) of smooth functions $\{\phi_Q\}_{Q\in \mathcal{Q}^\alpha}$ and 
$\{\tilde{\phi}_Q\}_{Q\in \mathcal{Q}^\alpha}$
satisfying $\text{supp}(\phi_Q)\subseteq Q^e$,
$$\sum_{Q\in \mathcal{Q}^\alpha}\phi_Q(\xi)=\sum_{Q\in \mathcal{Q}^\alpha}\tilde{\phi}_Q(\xi)=  1, \qquad \xi\in \bR^2\backslash \{0\},$$
and $\tilde{\phi}_Q(x)=1$ for $x\in \supp(\phi_Q)$, $Q\in \mathcal{Q}^\alpha$. Moreover,
one can ensure that the sequences  $\{\phi_Q\}$ and $\{\tilde{\phi}_Q\}$ act as  bounded multiplier sequences on certain vector-valued $L_p$-spaces as stated in Proposition \ref{th:vecmult}.  We will not discuss this rather technical issue here, but instead refer the reader to the discussion in \cite{AlJawahri2019,Borup2008}, and henceforth assume that $\{\phi_Q\}$ and $\{\tilde{\phi}_Q\}$ both are constructed such that they satisfy the multiplier condition needed for Proposition \ref{th:vecmult} for any $0<p<\infty$ and $0<q\leq \infty$.

We can now define the homogeneous (anisotropic) T-L type spaces and the decomposition spaces. We let $ \mathcal{S}'\backslash\mathcal{P}$ denote the class of tempered distributions modulo polynomials defined on $\mathbb{R}^2$.

\begin{mydef}\label{def:complete}
Let $\tilde{h}_\alpha$ be a hybrid weight for $\mathcal{Q}^\alpha$. Let $\{\phi_j\}_{j\in J}$ be a corresponding BAPU and set $\phi_j(D)f := \mathcal{F}^{-1}(\phi_j\mathcal{F}f)$.
\begin{itemize}
	\item For $s\in\R, 0<p<\infty$ and $0<q\leq \infty$, we define the (anisotropic) homogeneous Triebel-Lizorkin space $\dot{F}_{p,q}^{s,\alpha}(\bR^2)$ as the set all $f\in \mathcal{S}'\backslash\mathcal{P}$ satisfying
	\begin{equation*}
	\norm{f}_{\dot{F}_{p,q}^{s,\alpha}(\bR^2)} := \norm{\left( \sum_{Q\in \mathcal{Q}^\alpha} |\tilde{h}_{\alpha}(\xi_Q)^s\phi_Q(D)f|^q\right)^{1/q}}_{L_p} < \infty.
	\end{equation*}
	\item For $s \in \mathbb{R}, 0<p\leq \infty$ and $0<q<\infty$ we define the (anisotropic) homogeneous decomposition space $\dot{M}_{p,q}^{s,\alpha}(\bR^2)$ as the set of all $f\in\mathcal{S}'\backslash\mathcal{P}$ satisfying
	\begin{equation*}
	\norm{f}_{\dot{M}_{p,q}^{s,\alpha}(\bR^2)}  =
	\left( \sum_{Q\in \mathcal{Q}^\alpha} \norm{\tilde{h}(\xi_Q)^s \phi_Q(D)f}_{L_p}^q \right)^{1/q} < \infty,
	\end{equation*}
	with the usual modification if $q=\infty$.
\end{itemize}
\end{mydef}

It can be verified that $\dot{F}_{p,q}^{s,\alpha}(\bR^2)$ and  $\dot{M}_{p,q}^{s,\alpha}(\bR^2)$ are quasi-Banach spaces if  $0<p<1$ or $0<q<1$, and they are Banach spaces when $1\leq p,q <\infty$. The particular space does not (up to norm equivalence)  depend on the choice of $\rho$ in the definition of $\tilde{h}_\alpha$ nor does it depend on the particular choice of sample frequencies $\xi_Q$ as long as $\xi_Q\in Q$, see the moderation condition \eqref{eq:moderation}, and it does not depend on the particular choice of BAPU, see \cite{AlJawahri2019,Borup2008}. In particular,  $\{\tilde{\phi}_Q\}_Q$ will generate the same spaces up to norm equivalence. We  mention that it is possible to consider other reservoirs of distributions  than  $\mathcal{S}'\backslash\mathcal{P}$ to build the function spaces, see Voigtlaender \cite{Voigtlaender2016a} for further details. 

Let us recall that class $\mathcal{S}_0:=\mathcal{S}_0(\mathbb{R}^2)$, which is the closed subspace of the Schwartz class $\mathcal{S}(\mathbb{R}^2)$, is defined by
\begin{align*}
\mathcal{S}_0&=
\left\{f\in\mathcal{S}(\mathbb{R}^2): \int f(x)\cdot x^\alpha\,dx=0\, \text{ for all } \alpha\in\N_0^2\right\}.
\end{align*}

It can be proved
that $\dot{F}_{p,q}^{s,\alpha}(\bR^2)$ $[\dot{M}_{p,q}^{s,\alpha}(\bR^2)]$ satisfy
$$\mathcal{S}_0\hookrightarrow \dot{M}_{p,q}^{s,\alpha}(\bR^2)\hookrightarrow
\cS'\backslash\mathcal{P},\qquad \cS_0\hookrightarrow \dot{F}_{p,q}^{s,\alpha}(\bR^2)\hookrightarrow
\cS'\backslash\mathcal{P},$$ see \cite{AlJawahri2019}. Moreover, if $p,q<\infty$, $\mathcal{S}_0$ is dense in
$M^{s,\alpha}_{p,q}(\bR^2)$.

\subsection{A characterization of \texorpdfstring{$\dot{F}_{p,q}^{s,\alpha}(\bR^2)$}{Fpq}}
We claim that the spaces $\dot{F}_{p,q}^{s,\alpha}(\bR^2)$ $[\dot{M}_{p,q}^{s,\alpha}(\bR^2)]$ can be completely characterized using the brushlet system build on $\mathcal{Q}^\alpha$. In this note, we focus on proving this claim for the Triebel-Lizorkin type spaces $\dot{F}_{p,q}^{s,\alpha}(\bR^2)$ spaces. For the many of the proofs in this Section, we will call on results on vector-valued multiplies that can be found in Appendix \ref{s:app}. The modulation spaces $\dot{M}_{p,q}^{s,\alpha}(\bR^2)$ are easier to handle due to their (simpler) structure, and the reader can verify that the results and proofs presented in \cite{Nielsen2010} can be adapted to this homogeneous setup.

For $Q=I\times J\in \cup_j \cA_j$, we defined an associated dilation matrix by $\delta_Q:=\text{diag}(|I|,|J|)$. We define for $\cn\in \bN_0^2$,

\begin{equation}
  \label{eq:QD}
  U(Q,\cn)=\left\{y\in\bR^2:
\delta_{Q} y-\pi{\left(\cn+\ca\right)}\in B(0,1)\right\},
\end{equation}
where $\ca:=[\frac12,\frac12]^T$.
It is easy to verify there exists $L<\infty$ so that uniformly in $x$
and $Q$,  $\sum_n\chi_{U(Q,\cn)}(x)\leq L$. One may also verify that for $\cn,\cn'\in \bN_0^2$, $U(Q,\cn')=U(Q,\cn)+\pi\delta_Q^{-1}(\cn'-\cn)$.

We can now prove that the canonical coefficient operator is bounded on
$F^s_{p,q}(h,w)$.

\begin{mylemma}\label{prop:normchar}
Let $\{T_Q=\delta_Q\cdot+\xi_Q\}_{Q\in\mathcal{Q}^\alpha}$ be the family of
invertible affine
  transformations associated with $\mathcal{Q}^\alpha$ in \eqref{eq:affine}.
Suppose $s\in\bR$, $0<p< \infty$, and $0<q\leq \infty$. Then
$$\|\mathcal{S}_q^s(f)\|_{L_p}\leq C\|f\|_{\dot{F}^{s,\alpha}_{p,q}},\qquad f\in \dot{F}^{s,\alpha}_{p,q}(\R^2),$$
where
\begin{equation}
    \label{eq:Sq}
\mathcal{S}_q^s(f):=\bigg(\sum_{Q}\sum_{\cn\in\bN_0^2}(\tilde{h}_{\alpha}(\xi_Q)^s|\langle
f,w_{\cn,Q}\rangle_A||Q|^{1/2}\chi_{U(Q,\cn)})^q \bigg)^{1/p},
    \end{equation}
with $U(Q,\cn)$ given in \eqref{eq:QD}. 
\end{mylemma}

\begin{myre}
Using the observation in \eqref{eq:Qrule}, we also have 
$$\mathcal{S}_q^s(\cdot)\asymp \bigg(\sum_{Q}\sum_{\cn\in\bN_0^2}(\tilde{h}_{\alpha}(\xi_Q)^{s+1}|\langle\,
\cdot \,,w_{\cn,Q}\rangle_A|\chi_{U(Q,\cn)})^q \bigg)^{1/p}.$$
\end{myre}

\begin{proof}
Take $f\in  F^{s,\alpha}_{p,q}(\R^2)$ and fix  $Q\in \mathcal{Q}^\alpha$. We write the cosine term in Eq.\ \eqref{eq:brush1} as a sum of complex exponentials, and we take a tensor product to create $w_{\cn,Q}$ . This process creates a bi-variate function with four "humps", and, as it turns out, we will consequently need four terms to control the inner product $\langle
f,w_{\cn,Q}\rangle$. We first obtain the estimate
\begin{align*}|\langle
f,w_{\cn,Q}\rangle|&\leq \sqrt{\frac{2}{|Q|}}\sum_{j=1}^4\big|(b_Q(D)f)(\cv_j)\big|,
\end{align*}
with $b_Q(D)$ defined in \eqref{eq:S}, $\cv_1:=\pi 
\delta_{Q}^{-1} (\cn+\ca)$, $\cv_2:=-\cv_1$, $\cv_3:=\tilde{\cv}_1$, $\cv_4:=\tilde{\cv}_2$, where for a vector $\cv=[v_1,v_2]^T$, we let 
 $\tilde{\cv}:=[v_1,-v_2]^T$.
Notice, if $U(Q,\cn)\cap
U(Q,\cn')\not=\emptyset$ and 
$u\in U(Q,\cn), v\in U(Q,\cn')$ then $|u-v|\leq c|Q|^{-1/2}$ for some $c>0$ independent of $Q$.
Using the observations about the sets $U(Q,\cn)$ above, and defining the linear maps $R_j:\bR^2\rightarrow \bR^2$, $j=1,\ldots,4$,  by $R_1:=\text{Id}$, $R_2:=-R_1$, $R_3u:=\tilde{u}$, $u\in\bR^2$, and $R_4:=-R_3$, we obtain

\begin{align*}
\big|(b_Q(D)f)(\cv_1)\big|&\leq \sup_{y\in U(Q,\cn)}\big|(b_Q(D)f)(y)\big|
\leq \sup_{u\in B(0,c|Q|^{-1/2})}\big|(b_Q(D)f)(R_1x-u)\big|
,\\
\big|(b_Q(D)f)(\cv_2)\big|&\leq \sup_{y\in U(Q,-\cn-2\ca)}\big|(b_Q(D)f)(y)\big|\leq \sup_{u\in B(0,c|Q|^{-1/2})}\big|(b_Q(D)f)(R_2x-u)\big|,\\
\big|(b_Q(D)f)(\cv_3)\big|&\leq \sup_{y\in U(Q,\tilde{\cn}+\ce_2)}\big|(b_Q(D)f)(y)\big|\leq \sup_{u\in B(0,c|Q|^{-1/2})}\big|(b_Q(D)f)(R_3x-u)\big|,\\
\big|(b_Q(D)f)(\cv_4)\big|&\leq \sup_{y\in U(Q,-\tilde{\cn}-\ce_1)}\big|(b_Q(D)f)(y)\big|\leq \sup_{u\in B(0,c|Q|^{-1/2})}\big|(b_Q(D)f)(R_4x-u)\big|.
 \end{align*}

We now estimate the inner sum in $\mathcal{S}_q^{s}(f)$ to obtain,
 \begin{align*}
\sum_{\cn\in\bN_0^2}(|\langle
f,w_{\cn,Q}\rangle|&|Q|^{1/2}\chi_{U(Q,\cn)}(x))^q\\
&\leq 
C\sum_{\cn\in\bN_0^2}\sum_{j=1}^4 (
|(b_Q(D)f)(\cv_j)|\chi_{Q(k,n)}(x))^q\\
&\leq C\sum_{j=1}^4\sup_{u\in {B}(0,c|Q|^{-1/2})}\big(\brac{\delta_{Q} u}^{-2/r} 
|(b_Q(D)f)(R_jx-u)|\big)^q\cdot \brac{\delta_{Q} u}^{2 q/r}\\
&\leq C \sum_{j=1}^4\big(\sup_{u\in\bR^2}\brac{\delta_{Q}^\top u}^{-2/r} 
|(b_Q(D)f)(R_jx-u)|\big)^q\\
&\leq C \sum_{j=1}^4 (b_Q(D)f)^*(2/r,c|Q|^{1/2};R_j x)^q.
    \end{align*}
 Recall  that $\supp(\widehat{b_Q(D)f})\subset T_Q([-0.6,0.6]^2)$, so by 
Proposition \ref{prop1}, Proposition \ref{th:vecmult}, and the estimate above,
\begin{align*}
  \|\mathcal{S}_q^{s}(f)\|_{L_p}&\leq
  C\bigg\|\bigg(\sum_{Q}(\tilde{h}_{\alpha}(\xi_Q)^{sq}\sum_{j=1}^4 (b_Q(D)f)^*(2/r,c|Q|^{1/2};R_j x) )^q \bigg)^{1/p}\bigg\|_{L_p}\\ 
&\leq C'
\sum_{j=1}^4  \bigg\|\bigg(\sum_{Q}(\tilde{h}_{\alpha}(\xi_Q)^{sq} (b_Q(D)f)^*(2/r,c|Q|^{1/2};R_j x) )^q \bigg)^{1/p}\bigg\|_{L_p}\\
&\leq 4C'  \bigg\|\bigg(\sum_{Q}(\tilde{h}_{\alpha}(\xi_Q)^{sq} (b_Q(D)f)^*(2/r,c|Q|^{1/2};x) )^q \bigg)^{1/p}\bigg\|_{L_p}\\
&= 4C'  \bigg\|\bigg(\sum_{Q}(\tilde{h}_{\alpha}(\xi_Q)^{sq} (b_Q(D)\tilde{\phi}_Q(D)f)^*(2/r,c|Q|^{1/2};x) )^q \bigg)^{1/p}\bigg\|_{L_p}\\
&\leq C''  \bigg\|\bigg(\sum_{Q}(\tilde{h}_{\alpha}(\xi_Q)^{sq} (\tilde{\phi}_Q(D)f)(x) )^q \bigg)^{1/p}\bigg\|_{L_p}\\
&=C''\|f\|_{\dot{F}^{s,\alpha}_{p,q}}.
\end{align*}
\end{proof}

Inspired by Lemma \ref{prop:normchar}, we define the sequence space
$\dot{f}^{s,\alpha}_{p,q}:=\dot{f}^{s,\alpha}_{p,q}(\bR^2)$ for $s\in \bR$, $0<p<\infty$, and
$0<q\leq \infty$ , as the set of sequences 
$\{s_{Q,n}\}_{Q\in Q,n\in\bZ^d}\subset \mathbb{C}$ satisfying
$$\|\{s_{Q,n}\}\|_{\dot{f}^{s,\alpha}_{p,q}}:= \bigg\| \bigg\{\tilde{h}_{\alpha}(\xi_Q)^s
|Q|^{1/2}\Bigl(\sum_{n\in \bZ^d}|s_{Q,n}|^q
\chi_{U(Q,\cn)}\Bigr)^{1/q}\bigg\}_k\bigg\|_{L_p(\ell_q)}<\infty,$$
where the $L_p(\ell_q)$-norm is defined for  a sequence
$f=\{f_j\}_{j\in\bN}$ of measurable functions by
$$\|f\|_{L_p(\ell_q)}:=\big\|\big(\sum_{j\in\bN} 
|f_j|^q\big)^{1/q}\big\|_{L_p(\bR^2)},$$ 
see also Appendix \ref{s:app}.
Lemma \ref{prop:normchar} provides us with a bounded coefficient
operator $C\colon F_{p,q}^{s,\alpha}\rightarrow f^{s,\alpha}_{p,q}$ given by
\begin{equation}
  \label{eq:coefficient}
  Cf=\{\langle f,w_{n,Q}\rangle\}_{Q\in \cQ^\alpha,\cn\in\bN_0^2}.
\end{equation}
Moreover, the fact that $\{w_{n,Q}\}$ is an orthonormal basis shows that
the only consistent  definition of a  reconstruction operator is given by
\begin{equation}
  \label{eq:reconstruct}
 R: \{s_{Q,\cn}\}_{Q,n}\rightarrow \sum_{Q,\cn} s_{Q,\cn}
w_{n,Q}.
\end{equation}
Using Lemma \ref{lem:maxbound} we now verify 
that $R:f^{s,\alpha}_{p,q}\rightarrow F^{s,\alpha}_{p,q}$ is also a bounded operator.

\begin{mylemma}\label{lem:reconstruct}
  Suppose $0\leq \alpha<1$, $s\in \bR$, $0<p< \infty$, and $0<q\leq \infty$. Then for
  any finite sequence
  $\{s_{Q,\cn}\}_{Q,\cn}$, we have 
  $$\Bigl\|\sum_{Q,\cn} s_{Q,\cn}w_{n,Q}\Bigr\|_{\dot{F}_{p,q}^{s,\alpha}} \leq C
  \|\{s_{Q,\cn}\}\|_{\dot{f}^{s,\alpha}_{p,q}}.$$
\end{mylemma}

\begin{proof}
Let $\{\phi_Q\}_{Q\in \cQ}$ be the BAPU associated with
$\mathcal{Q}^\alpha$. Using the structure given by \eqref{eq:brush1}, and Proposition \ref{th:vecmult}, we get
  \begin{align*}
    \Bigl\|\sum_{Q,\cn} s_{Q,\cn}w_{n,Q}\Bigr\|_{\dot{F}_{p,q}^{s,\alpha}} &= \Bigl\|
    \Big\{w(\xi_{Q})^s \phi_{Q}(D)\Big(\sum_{Q',\cn}
    s_{Q',\cn}w_{Q',\cn}\Big)\Big\}_{Q}\Bigr\|_{L_p(\ell_q)}\\
    &\leq C\Bigl\|\Big\{
    \tilde{h}_{\alpha}(\xi_Q)^s \sum_{Q' \in N(Q)} \sum_{\cn}
    s_{Q',\cn}w_{Q',\cn}\Big\}_k\Bigr\|_{L_p(\ell_q)},
\end{align*}
where $N(Q)= \{Q'\in \cQ\colon \supp (\phi_Q)\cap \supp
(b_{Q'})\neq \emptyset\}$. It follows from  \cite[Lemma 2.8]{AlJawahri2019} that $\# N(Q)$ is
uniformly bounded, and since $\tilde{h}_\alpha$ is a moderate weight, see \eqref{eq:moderation},
we obtain 
$$\Bigl\|
   \Big\{ \tilde{h}_{\alpha}(\xi_Q)^s \sum_{Q' \in N(Q)} \sum_{\cn}
    s_{Q',\cn}w_{Q',\cn}\Big\}_Q
\Bigr\|_{L_p(\ell_q)}\leq
    C\biggl\|\biggl( \sum_{Q'} \Bigl( w(\xi_{Q'})^s \sum_n
    |s_{Q',\cn}| |w_{Q',\cn}|\Bigr)^q\biggr)^{1/q}\biggr\|_{L_p}.$$
Fix $0<r<\min(1,p,q)$. Then Lemma \ref{lem:maxbound} and the
Fefferman-Stein maximal inequality \eqref{eq:fs} yields
\begin{align*}
    \Bigl\|\Big\{
    \tilde{h}_{\alpha}(\xi_Q)^s \sum_{\cn} |s_{Q,\cn}| &|\eta_{Q,\cn}|\Big\}_Q
\Bigr\|_{L_p(\ell_q)}\\
    &\leq C\Bigl\|\Big\{
    \tilde{h}_{\alpha}(\xi_Q)^s |Q|^{1/2} \sum_{\ell=1}^4 M_r\Bigl(\sum_{\cn}
  |s_{Q,\cn}|
  \chi_{U(Q,\cn)}\Bigr)(R_{\ell}\cdot)\Big\}_Q
\Bigr\|_{L_p(\ell_q)}\\
&\leq C'\Bigl\|\Big\{
    \tilde{h}_{\alpha}(\xi_Q)^s |Q|^{1/2} \sum_{\cn}
  |s_{Q,\cn}| \chi_{U(Q,\cn)}\Big\}_Q\Bigr\|_{L_p(\ell_q)},
  \end{align*}
  where we used the (quasi-)triangle inequality and straightforward substitutions in the integrals.
The result now follows since the sum over $\cn$ is locally finite with a
uniform bound on the number of non-zero terms, which implies that
$$\Big(\sum_{\cn}
  |s_{k,\cn}| \chi_{U(Q,\cn)}\Big)^q\asymp \sum_{\cn}
  |s_{Q,\cn}|^q \chi_{U(Q,n)},$$
uniformly in $Q$.
\end{proof}

We now use Lemma \ref{prop:normchar} and Lemma
\ref{lem:reconstruct} to obtain the main result of this paper, that $\{w_{\cn,Q}\}$ forms
captures the norm of $\dot{F}^{s,\alpha}_{p,q}$, and forms an unconditional basis for  $\dot{F}^{s,\alpha}_{p,q}$ in the Banach space case.

\begin{mytheorem}\label{prop:modu}
Let $s\in \bR$, $0<p,q< \infty$. Then we have the norm characterization 
$$ \|f\|_{\dot{F}^{s,\alpha}_{p,q}}\asymp  \|\mathcal{S}_q^{s}(f)\|_{L_p},$$
with $\mathcal{S}_q^{s}(f)$ given by \eqref{eq:Sq}. 
Moreover, for $1\leq p,q<\infty$,  $\{w_{\cn,Q}\}$ forms an unconditional basis for $\mathcal{S}_q^{s}(f)$.
\end{mytheorem}

\begin{proof}
  The norm characterization follows at once by combining Lemma \ref{prop:normchar} and Lemma \ref{lem:reconstruct}. The claim that the system forms an unconditional
basis when $1\leq p, q < \infty$ follows easily from the fact that  $\dot{F}^{s,\alpha}_{p,q}$ is a Banach space, and
that finite expansions in  $\{w_{\cn,Q}\}$ have uniquely determined coefficients giving us a norm characterization of such expansions using the $L_p$-norm of  $\mathcal{S}_q^{s}(\cdot)$.
\end{proof}

We conclude with a few remarks on  Theorem \ref{prop:modu}. 
\begin{itemize}
    \item[a.]
A  result similar to Theorem \ref{prop:modu}, but with a much simplified proof, holds true for the homogeneous $\alpha$-modulation spaces $\dot{M}^{s,\alpha}_{p,q}(\R^2)$. For this case, one can follow the approach  in \cite{Nielsen2010,AlJawahri2019}.
\item[b.] The norm characterisation obtained in Theorem \ref{prop:modu} may appear similar to the characterisation obtained for tight frames in \cite[Theorem 7.5]{AlJawahri2019}, but one should notice the important additional  fact that $\{w_{\cn,Q}\}$ forms an unconditional basis. This fact has significant implications for, e.g., $n$-term nonlinear approximation from $\{w_{\cn,Q}\}$, where the linear independence will allow one to prove inverse estimates of Bernstein type. Inverse estimates are currently out of reach for the redundant frames considered in \cite{AlJawahri2019}, see the discussion of this problem in \cite{MR1794807}. Approximation properties of   $\{w_{\cn,Q}\}$ will be considered in a future publication.
\end{itemize}

\appendix
\section{Some technical results}\label{s:app}
This appendix contains some results on vector-valued maximal functions needed for the analysis of the $\alpha$-TL spaces. 

For $0<r<\infty$, the Hardy-Littlewood maximal function is defined by
$$M_r u(x):=\sup_{t>0}\bigg(\frac 1 {|B(x,t)|}
\int_{{B}(x,t)} |u(y)|^r dy\bigg)^{1/r},\qquad u\in
L_{r,\text{loc}}(\bR^2).$$

For $0<p,q\leq \infty$, and  a sequence
$f=\{f_j\}_{j\in\bN}$ of $L_p(\bR^2)$ functions, we define the norm
$$\|f\|_{L_p(\ell_q)}:=\big\|\big(\sum_{j\in\bN} 
|f_j|^q\big)^{1/q}\big\|_{L_p(\bR^2)}.$$ 
Where there is no risk of ambiguity we will abuse notation and write $\|f_k\|_{L_p(\ell_q)}$ instead
of $\|\{f_k\}_k\|_{L_p(\ell_q)}$.

The vector-valued Fefferman-Stein maximal inequality gives the estimate (see
\cite[Chapters I\&II]{Stein1993})  
\begin{equation}
  \label{eq:fs}
  \|\{M_rf_j\}\|_{L_p(\ell_q)}\leq C_B\|\{f_j\}\|_{L_p(\ell_q)}
\end{equation}
 for  $r<q\leq \infty$ and $r<p<\infty$, $C_B:=C_B(r,p,q)$.

For $\Omega=\{\Omega_n\}$ a sequence of compact
subsets of $\bR^2$, we let
$$L_p^\Omega(\ell_q):=\{\{f_n\}_{n\in\bN}\in
L_p(\ell_q)\,|\,\supp(\hat{f}_n)\subseteq\Omega_n,\,\forall n\}.$$ 
For $\xi\in\bR^2$, we let $\brac{\xi}:=(1+|\xi|^2)^{1/2}$. Let $u(x)$ be a continuous function on $\bR^2$. We define, for $a,R>0$,
$$u^*(a,R;x):=\sup_{y\in\bR^2}
\brac{y}^{-a}|u(x-y/R)|,\qquad x\in\bR^2.$$

The following is a variation on Peetre's maximal estimate in a vector-valued setting.

\begin{myprop}\label{prop1}
    Suppose $0<p<\infty$ and
  $0<q\leq \infty$, and let $\Omega=\{T_k \mathcal{C}\}_{k\in\bN}$ be
  a sequence of compact subsets of $\bR^2$ generated by a family
  $\{T_k={t_k}Id \cdot+\xi_k\}_{k\in\bN}$ of invertible affine
  transformations on $\bR^2$, with $\mathcal{C}$ a fixed compact subset of
  $\bR^2$. If $0<r<\min(p,q)$, then there exists a constant $K$ such
  that
  \begin{equation}
    \label{eq:max1}
    \bigg\|\big\{ (f_k)^*(2/r,t_k;\cdot)
\big\}
\bigg\|_{L_p(\ell_q)}\leq K\|\{f_k\}\|_{L_p(\ell_q)},
  \end{equation}
for all $f\in L_p^\Omega(\ell_q)$, where $f=\{f_k\}_{k\in\bN}.$
\end{myprop}

Finally, we need the following vector-valued multiplier result.
For $s\in \bR_+$, we let
$$\|f\|_{H_2^s}:=\biggl( \int |\mathcal{F}^{-1}f(x)|^2 \brac{x}^{2s}
dx\biggr)^{1/2}$$
denote the Sobolev  norm.

\begin{myprop}\label{th:vecmult}
  Suppose $0<p<\infty$ and
  $0<q\leq \infty$, and let $\Omega=\{T_k \mathcal{C}\}_{k\in\bN}$ be
  a sequence of compact subsets of $\bR^2$ generated by a family
  $\{T_k={t_k}Id \cdot+\xi_k\}_{k\in\bN}$ of invertible affine
  transformations on $\bR^2$, with $\mathcal{C}$ a fixed compact subset of
  $\bR^2$. Assume $\{\psi_j\}_{j\in \bN}$ is a
  sequence of functions satisfying $\psi_j\in H^s_2$ for some $s>\frac
  \nu2+\frac{\nu}{\min(p,q)}$. Then there exists a constant $C<\infty$
  such that 
$$\|\{\psi_k(D)f_k\}\|_{L_p(\ell_q)} \leq
C\sup_j\|\psi_j(T_j\cdot)\|_{H^s_2}\cdot
\|\{f_k\}\|_{L_p(\ell_q)}$$
for all $\{f_k\}_{k\in \bN} \in L_p^\Omega(\ell_q)$.
\end{myprop}

The following Lemma was used in the proof of Lemma \ref{lem:reconstruct}.
\begin{mylemma}\label{lem:maxbound}
  Let $0<r\leq 1$. There exists a constant $C$ such that for any sequence $\{s_{Q,\cn}\}_{Q,\cn}$ we have
  $$\sum_{\cn} |s_{Q,\cn}||w_{n,Q}|(x)\leq C|Q|^{1/2} \sum_{\ell=1}^4 M_r\Bigl(\sum_{\cn}
  |s_{Q,\cn}| \chi_{U(Q,\cn)}\Bigr)(R_\ell x).$$
\end{mylemma}

\begin{proof}
  From \eqref{eq:gdecay} we have that 
  \begin{equation}\label{eq:west}|w_{n,Q}(x)|\leq
  C_N|Q|^{1/2}\sum_{\ell=1}^4\big(1+\big|
R_\ell\delta_{Q}x-\pi(\cn+\ca)\big|\big)^{-N},
\end{equation} 
for any $N>0$, with $C_N$ independent of $Q$,
where we use the same notation as in the proof of Lemma \ref{lem:reconstruct}.
 Fix $N>2/r$. We can,
  without loss of generality, suppose $x\in
  U(Q,\mathbf{0})$. 
    For $j\in \bN$, we let $A_j=\{ \cn\in \bN_0^2\colon 2^{j-1}<| \pi(\cn+\ca)|\leq 2^j\}$. Notice that $\cup_{\cn\in A_j}
    U(Q,\cn)$ is a bounded set contained in the ball
    ${B}(0,c2^{j+1}|Q|^{-1/2})$. Now,
    \begin{align*}
      \sum_{\cn\in A_j} |s_{Q,\cn}|\big(1+\big| 
\delta_{Q}x-&\pi(\cn+\ca)\big|\big)^{-N} \\&\leq C2^{-jN} \sum_{\cn\in
        A_j} |s_{Q,\cn}|\\
&      \leq C2^{-jN} \Bigl( \sum_{\cn\in
        A_j} |s_{Q,\cn}|^r\Bigr)^{1/r}\\
      &\leq C2^{-jN} |Q|^{1/r} \biggl( \int \sum_{\cn\in
        A_j} |s_{Q,\cn}|^r\chi_{U(Q,\cn)}(y)\, dy \biggr)^{1/r}\\
      &\leq CL^{1-r}2^{-jN} |Q|^{1/r} \biggl( 
\int_{{B}(0,c2^{j+1}|Q|^{-1/2})} \Bigl(\sum_{\cn\in
        A_j} |s_{Q,\cn}|\chi_{U(Q,\cn)}(y)\Bigr)^r \, dy\biggr)^{1/r}\\
      &\leq C'2^{-j(N-2/r)} M_r \Bigl(\sum_{\cn\in
       \bN_0^2} |s_{Q,\cn}|\chi_{U(Q,\cn)} \Bigr)(x).
    \end{align*}
We now perform the summation over $j\in \bN_0$ to obtain
$$    \sum_{\cn\in \bN_0^2} |s_{Q,\cn}|\big(1+\big| 
\delta_{Q}x-\pi(\cn+\ca)\big|\big)^{-N} 
\leq  CM_r \Bigl(\sum_{\cn\in
       \bN_0^2} |s_{Q,\cn}|\chi_{U(Q,\cn)} \Bigr)(x).$$
We then use the substitutions $x=R_\ell z$, $\ell=1,\ldots,4$, to cover all four terms on the RHS of \eqref{eq:west}, where we use the fact that $R_{\ell}$ and $\delta_Q$ commute. 
\end{proof}

\bibliographystyle{abbrv}

\begin{thebibliography}{10}

\bibitem{AlJawahri2019}
Z.~Al-Jawahri and M.~Nielsen.
\newblock On homogeneous decomposition spaces and associated decompositions of
  distribution spaces.
\newblock {\em Mathematische Nachrichten}, 292(12):2496--2521, 2019.

\bibitem{Auscher1992}
P.~Auscher, G.~Weiss, and M.~V. Wickerhauser.
\newblock Local sine and cosine bases of {C}oifman and {M}eyer and the
  construction of smooth wavelets.
\newblock In {\em Wavelets}, volume~2 of {\em Wavelet Anal. Appl.}, pages
  237--256. Academic Press, Boston, MA, 1992.

\bibitem{Borup2003}
L.~Borup and M.~Nielsen.
\newblock Approximation with brushlet systems.
\newblock {\em J. Approx. Theory}, 123(1):25--51, 2003.

\bibitem{Borup2006b}
L.~Borup and M.~Nielsen.
\newblock Nonlinear approximation in {$\alpha$}-modulation spaces.
\newblock {\em Math. Nachr.}, 279(1-2):101--120, 2006.

\bibitem{Borup2007}
L.~Borup and M.~Nielsen.
\newblock Frame decomposition of decomposition spaces.
\newblock {\em J. Fourier Anal. Appl.}, 13(1):39--70, 2007.

\bibitem{Borup2008}
L.~Borup and M.~Nielsen.
\newblock On anisotropic {T}riebel-{L}izorkin type spaces, with applications to
  the study of pseudo-differential operators.
\newblock {\em J. Funct. Spaces Appl.}, 6(2):107--154, 2008.

\bibitem{DeVore1992}
R.~A. DeVore, B.~Jawerth, and B.~J. Lucier.
\newblock Image compression through wavelet transform coding.
\newblock {\em IEEE Trans. Inform. Theory}, 38(2, part 2):719--746, 1992.

\bibitem{DeVore1992a}
R.~A. DeVore, B.~Jawerth, and V.~Popov.
\newblock Compression of wavelet decompositions.
\newblock {\em Amer. J. Math.}, 114(4):737--785, 1992.

\bibitem{Feichtinger1987}
H.~G. Feichtinger.
\newblock Banach spaces of distributions defined by decomposition methods.
  {II}.
\newblock {\em Math. Nachr.}, 132:207--237, 1987.

\bibitem{Feichtinger1985}
H.~G. Feichtinger and P.~Gr{\"o}bner.
\newblock Banach spaces of distributions defined by decomposition methods. {I}.
\newblock {\em Math. Nachr.}, 123:97--120, 1985.

\bibitem{Gribonval2004}
R.~Gribonval and M.~Nielsen.
\newblock Nonlinear approximation with dictionaries. {I}. {D}irect estimates.
\newblock {\em J. Fourier Anal. Appl.}, 10(1):51--71, 2004.

\bibitem{Groebner1992}
P.~Gr\"obner.
\newblock {\em Banachr\"aume glatter {F}unktionen und {Z}erlegungsmethoden}.
\newblock PhD thesis, University of Vienna, 1992.

\bibitem{Hernandez1996}
E.~Hern{\'a}ndez and G.~Weiss.
\newblock {\em A first course on wavelets}.
\newblock Studies in Advanced Mathematics. CRC Press, Boca Raton, FL, 1996.
\newblock With a foreword by Yves Meyer.

\bibitem{Kyriazis2002}
G.~Kyriazis and P.~Petrushev.
\newblock New bases for {T}riebel-{L}izorkin and {B}esov spaces.
\newblock {\em Trans. Amer. Math. Soc.}, 354(2):749--776, 2002.

\bibitem{Laeng1990}
E.~Laeng.
\newblock Une base orthonormale de {$L\sp 2({\bf R})$} dont les \'el\'ements
  sont bien localis\'es dans l'espace de phase et leurs supports adapt\'es \`a
  toute partition sym\'etrique de l'espace des fr\'equences.
\newblock {\em C. R. Acad. Sci. Paris S\'er. I Math.}, 311(11):677--680, 1990.

\bibitem{Meyer1997a}
F.~G. Meyer and R.~R. Coifman.
\newblock Brushlets: a tool for directional image analysis and image
  compression.
\newblock {\em Appl. Comput. Harmon. Anal.}, 4(2):147--187, 1997.

\bibitem{Meyer1992}
Y.~Meyer.
\newblock {\em Wavelets and operators}, volume~37 of {\em Cambridge Studies in
  Advanced Mathematics}.
\newblock Cambridge University Press, Cambridge, 1992.

\bibitem{Nielsen2010}
M.~Nielsen.
\newblock Orthonormal bases for $\alpha$-modulation spaces.
\newblock {\em Collect. Math.}, 61(2):173--190, 2010.

\bibitem{Paeivaerinta1988}
L.~P{\"a}iv{\"a}rinta and E.~Somersalo.
\newblock A generalization of the {C}alder\'on-{V}aillancourt theorem to {$L\sp
  p$} and {$h\sp p$}.
\newblock {\em Math. Nachr.}, 138:145--156, 1988.

\bibitem{MR1794807}
P.~Petrushev.
\newblock Nonlinear approximation from dictionaries: some open problems:
  research problems 2001-1.
\newblock {\em Constr. Approx.}, 17(1):153--155, 2001.

\bibitem{Rasmussen2012}
K.~N. Rasmussen.
\newblock Orthonormal bases for anisotropic $\alpha$-modulation spaces.
\newblock {\em Collect. Math.}, 63(1):109--121, 2012.

\bibitem{Stein1993}
E.~M. Stein.
\newblock {\em Harmonic analysis: real-variable methods, orthogonality, and
  oscillatory integrals}, volume~43 of {\em Princeton Mathematical Series}.
\newblock Princeton University Press, Princeton, NJ, 1993.
\newblock With the assistance of Timothy S. Murphy, Monographs in Harmonic
  Analysis, III.

\bibitem{Triebel1983a}
H.~Triebel.
\newblock Modulation spaces on the {E}uclidean {$n$}-space.
\newblock {\em Z. Anal. Anwendungen}, 2(5):443--457, 1983.

\bibitem{Voigtlaender2016a}
F.~{Voigtlaender}.
\newblock {Embeddings of Decomposition Spaces into Sobolev and BV Spaces}.
\newblock {\em ArXiv e-prints}, Jan. 2016.

\end{thebibliography}

\end{document}